\documentclass{amsart}

\usepackage{amsmath,amssymb,amsfonts,enumerate,amsthm,graphicx,color}

\newcommand{\reg}{\textnormal{reg}\,}
\newcommand{\heit}{\textnormal{ht}\,}
\newcommand{\pd}{\textnormal{pd}\,}
\newtheorem{thm}{Theorem}[section]

\newtheorem{lem}[thm]{Lemma}
\newtheorem{prop}[thm]{Proposition}
\newtheorem{defn}[thm]{Definition}

\newtheorem{rem}[thm]{Remark}
\newtheorem{ques}[thm]{Question}

\numberwithin{equation}{section}

\begin{document}
\bibliographystyle{amsplain}

\author[M. Mahmoudi]{Mohammad Mahmoudi}

\address{Mohammad Mahmoudi\\Department of Mathematics, Science and Research Branch, Islamic
Azad University (IAU), Tehran, Iran.}\email{mahmoudi54@gmail.com}

\author[A. Mousivand]{Amir Mousivand}

\address{Amir Mousivand\\Department of Mathematics, Science and Research Branch, Islamic Azad
University (IAU), Tehran, Iran.}\email{amirmousivand@gmail.com}

\author[M. Crupi]{Marilena Crupi}

\address{Marilena Crupi\\ Dipartimento Di Matematica, Universita� di Messina, Viale Ferdinando Stagno d'Alcontres, 31 \\
98166 Messina, Italy. Fax number: +39 090 393502}\email{mcrupi@unime.it}

\author[G. Rinaldo]{Giancarlo Rinaldo}

\address{Giancarlo Rinaldo\\ Dipartimento Di Matematica, Universita� di Messina, Viale Ferdinando Stagno d'Alcontres, 31 \\
98166 Messina, Italy. Fax number: +39 090 393502}\email{rinaldo@dipmat.unime.it}

\author[N. Terai]{Naoki Terai}

\address{Naoki Terai\\Department of Mathematics, Faculty of Culture and Education, Saga University, Saga 840-8502, Japan.}\email{terai@cc.saga-u.ac.jp}

\author[S. Yassemi]{Siamak Yassemi}
\address{Siamak Yassemi\\Department of Mathematics, University of
Tehran, Tehran, Iran and School of Mathematics, Institute for
Research in Fundamental Sciences (IPM), Tehran
Iran.}\email{yassemi@ipm.ir}




\title[vertex decomposability and regularity of . . .]{vertex decomposability and regularity of\\ very well-covered graphs}

\begin{abstract} A graph $G$ is well-covered if it has no isolated vertices and all the maximal independent sets have the same cardinality. If furthermore  two times this cardinality is equal to $|V(G)|$, the graph $G$ is called very well-covered. The class of very well-covered graphs contains bipartite well-covered graphs.
Recently in \cite{CRT} it is shown that a very well-covered graph $G$ is Cohen-Macaulay
if and only if it is pure shellable. In this article we improve this result by showing that $G$ is Cohen-Macaulay
if and only if it is pure vertex decomposable. In addition, if $I(G)$ denotes the edge ideal of $G$, we show that the Castelnuovo-Mumford regularity of $R/I(G)$ is equal to the maximum number of pairwise 3-disjoint edges of $G$. This improves Kummini's result on unmixed bipartite graphs.
\end{abstract}

\maketitle

\section{Introduction} Let $G$ be a simple undirected graph with the vertex set
$V(G)=\{ x_1 ,  . . .  , x_n\}$ and edge set $E(G)$. By identifying
the vertex $x_i$ with the variable $x_i$ in the polynomial ring
$R=k[ x_1 ,  . . .  , x_n]$ over a field $k$, one can associate to
$G$ a square-free monomial ideal $I(G)$ generated by all
quadratic square-free monomials $x_i x_j$ where $\{x_i , x_j\}$
is an edge of $G$. The ideal $I(G)$ is called the edge ideal of
$G$. A graph $G$ is said to be (sequentially) Cohen-Macaulay over
$k$ if $R/I(G)$ is a (sequentially) Cohen-Macaulay ring. It is
known that a graph $G$ is Cohen-Macaulay if and only if it is
sequentially Cohen-Macaulay and unmixed, i.e. all its minimal
vertex covers have the same cardinality. A recent stream in
commutative algebra and algebraic combinatorics is to describe
the algebraic properties of the edge ideal $I(G)$ in terms of
combinatorial properties of $G$. A graph $G$ is called bipartite
if its vertex set can be divided into two disjoint sets $V_1$ and
$V_2$ such that every edge connects a vertex in $V_1$ to one in
$V_2$. Unmixed bipartite graphs and Cohen-Macaulay bipartite
graphs have been characterized nicely in terms of combinatorial
properties of $G$ (see \cite {HH} and \cite {V2}).

A simplicial complex $\Delta$ over a set of vertices $V=\{v_1 ,
 . . .  , v_n \}$ is a collection of subsets of $V$, with the
property that $\{v_i\}\in\Delta$ for all $i$, and if $F\in\Delta$,
then all subsets of $F$ are also in $\Delta$ (including the empty
set). An element of $\Delta$ is called a face of $\Delta$ and a
simplicial complex is called pure if all its facets (maximal
faces with respect to inclusion) have the same cardinality. Using
the Stanley-Reisner correspondence, one can associate to $G$ the
simplicial complex $\Delta_G$ where $I_{\Delta_G}=I(G)$. Note
that the faces of $\Delta_G$ are the independent sets of $G$.
Thus $F$ is a face of $\Delta_G$ if and only if there is no edge
of $G$ joining any two vertices of $F$. This simplicial complex
is called the independence complex of $G$. We call a graph $G$
vertex decomposable (shellable) if the simplicial complex
$\Delta_G$ is vertex decomposable (shellable). We have the
following implications:
$$\mbox {pure vertex~decomposable}\Longrightarrow \mbox {pure shellable}\Longrightarrow \mbox {Cohen-Macaulay}$$
and it is known that the above implications are strict.

A graph $G$ is well-covered if it has no isolated vertices and all the maximal independent sets have the same cardinality. If furthermore  two times this cardinality is equal to $|V(G)|$, the graph is called very well-covered. In \cite{GV1} it is shown that for a well-covered graph we have $2\heit(I(G))\geq |V(G)|$. Since the complement of any maximal independent set is a minimal vertex cover. It follows that a well-covered graph $G$ is very well-covered if and only if $2\heit(I(G))=|V(G)|$. The class of very well-covered graphs contains bipartite well-covered graphs.
In this article we consider the class of very well-covered graphs with $2n$ vertices. It is known that any graph
in this class has perfect matching (see \cite [Remark 2.2]{GV2}).
Hence we may assume:\\

$(\ast)~V(G)=X\cup Y, X\cap Y=\emptyset$, where $X=\{ x_1 ,  . . .
, x_n\}$ is a minimal vertex cover of $G$ and $Y=\{ y_1 ,  . . .  ,
y_n\}$ is a maximal independent set of $G$ such that $\{ x_1y_1 ,
 . . .  , x_ny_n\}\subset E(G)$. In fact we have the following:

\begin{prop} \label{pro:well}(\cite[Proposition 2.3]{CRT}) Let $G$ be a graph with $2n$ vertices which are
not isolated and with $\heit(I(G))=n$. We assume the conditions
$(\ast)$. Then $G$ is unmixed (very well-covered) if and only if the following conditions hold:
\begin{itemize}
\item[(i)] if $z_ix_j,y_jx_k\in E(G)$, then $z_ix_k\in E(G)$ for
distinct $i,j,k$ and for $z_i\in\{x_i,y_i\};$
\item[(ii)] if $x_iy_j\in E(G)$ then $x_ix_j\notin E(G).$
\end{itemize}
\end{prop}
Also it is shown  in \cite[Lemma 3.5]{CRT} that if $G$ is Cohen-Macaulay, then there exists a suitable simultaneous change
of labeling on both $\{x_i\}$ and $\{y_i\}$
(i.e., we relable  $(x_{i_1}, \ldots, x_{i_n})$ and $(y_{i_1}, \ldots,
y_{i_n})$ as  $(x_1, \ldots, x_n)$ and $(y_1, \ldots, y_n)$ at
the same time), such that\\

$(\ast \ast)~x_iy_j\in E(G)$ implies $i\leq j.$\\\\
On the other hand for any graph $G$ satisfying $(\ast)$ and $(\ast \ast)$ we have the following:

\begin{thm} \label{th:CM} (\cite[Theorem 3.6]{CRT}) Let $G$ be a graph with $2n$ vertices which are
not isolated and with $\heit(I(G))=n$. We assume the conditions
$(\ast)$ and $(\ast \ast)$. Then the following conditions are
equivalent:

\begin{itemize}
\item[(1)] $G$ is Cohen-Macaulay;
\item[(2)] $G$ is unmixed (very well-covered);
\item[(3)] The following conditions hold:
\begin{itemize}
\item[(i)] if $z_ix_j,y_jx_k\in E(G)$, then $z_ix_k\in E(G)$ for
distinct $i,j,k$ and for $z_i\in\{x_i,y_i\};$
\item[(ii)] if $x_iy_j\in E(G)$ then $x_ix_j\notin E(G).$
\end{itemize}
\end{itemize}
\end{thm}
In the next remark we restate Proposition \ref{pro:well} and Theorem \ref{th:CM} as we will use throughout the paper.
\begin{rem} \em \label{rem:equiv} The above theorem shows that if $G$ is a graph with $2n$ vertices which are
not isolated and with $\heit(I(G))=n$, then
$$\mbox{$G$ is unmixed (very well-covered) $\Longleftrightarrow$ $G$ satisfies ($\ast$), and (i), (ii) of Proposition \ref{pro:well}}$$
$$\mbox{ $G$ is Cohen-Macaulay $\Longleftrightarrow$ $G$ satisfies ($\ast$), ($\ast\, \ast$), and (i), (ii) of Proposition \ref{pro:well}}$$
\end{rem}

Class of very well-covered graphs contains unmixed bipartite graphs which have no isolated vertices.
Van Tuyl in \cite[Corollary 2.12]{VT} showed that for a bipartite graph $G$, we have:
$$\mbox {$G$ Cohen-Macaulay} \Longleftrightarrow \mbox {$G$ pure shellable}\Longleftrightarrow\mbox {$G$ pure vertex~decomposable}.$$
On the other hand, the regularity of unmixed bipartite graphs have been studied in Kummini's work \cite{Ku}. He showed that if $G$ is an  unmixed bipartite graph, then $\reg(R/I(G))$ is equal to $a(G)$, where $a(G)$ is the maximum number of pairwise 3-disjoint edges of $G$, see \cite[Theorem 1.1]{Ku}.\\
It is natural to think on generalization of Van-Tuyl and Kummini's  results to the class of very well-covered graphs.
Recently the authors in \cite{CRT} showed that a graph $G$ in this class is Cohen-Macaulay
if and only if it is pure shellable, see \cite[Theorem 4.1]{CRT}.

The main results of the paper are the following theorems. \\
{\bf Theorem A.} Let $G$ be a very well-covered graph. Then the following conditions are equivalent:
\begin{itemize}
\item[{\bf (1)}] $G$ is Cohen-Macaulay.
\item[{\bf (2)}] $G$ is pure shellable.
\item[{\bf (3)}] $G$ is pure vertex decomposable.\\
\end{itemize}
{\bf Theorem B.} Let $G$ be a very well-covered graph. Then $\reg(R/I(G))=a(G)$.\\\\
These results improve the results of \cite{CRT}, \cite{Ku}, and \cite{VT}. More precisely Theorem A improves the result of \cite[Theorem 4.1]{CRT} and \cite[Corollary 2.12]{VT}, and Theorem B is a generalization of \cite[Theorem 1.1]{Ku}. \\\\

\section{Basic definitions and notations}

In this section we recall all the definitions and properties we will
use throughout the paper.

\begin{defn}{\bf (Pairwise 3-disjoint set of edges)} Let $G$ be a graph. Two edges $xy$ and
$uv$ of $G$ are called 3-disjoint if the induced
subgraph of $G$ on $\{x,y,u,v\}$ consists of exactly two disjoint
edges. A set $\Gamma$ of edges of $G$ is called pairwise
3-disjoint set of edges if any two edges of $\Gamma$ are
3-disjoint.\end{defn}

The \textit{maximum cardinality} of all pairwise 3-disjoint
sets of edges in $G$ is denoted by $a(G)$.\\\\
For a set $F\subseteq\{x_1, . . . ,x_n\}$, let
$$x_F=\prod _{x_i\in F} x_i.$$

\begin{defn}\label{def:cover}{\bf (Cover ideal of a graph)} Let $G$ be a graph over the vertex set $V(G)=\{x_1, . . . ,x_n\}$.  The cover ideal of $G$ denoted by
$I(G)^\vee$ is the square-free monomial ideal
$$I(G)^\vee=(x_F~|~ F~is~a~minimal~vertex~cover~of~G~).$$
\end{defn}

Let M be an arbitrary graded R-module, and let
$$0\rightarrow\bigoplus_j{R(-j)^{\beta_{t,j}(M)}}\rightarrow\bigoplus_j{R(-j)^{\beta_{t-1,j}(M)}}\rightarrow
\cdots\rightarrow\bigoplus_j{R(-j)^{\beta_{0,j}(M)}}\rightarrow
M\rightarrow 0$$
be the unique minimal graded free resolution of M over R,
where $R(-j)$ is a graded free R-module whose $n$-th graded
component is given by $R_{n-j}$.

The number $\beta_{i,j}(M)$ is called the $ij$-th graded Betti number of $M$ and it is equal to the
number of generators of degree $j$ in the $i$-th syzygy module.

The {\bf Castelnuovo-Mumford regularity of M}, denoted by $\reg(M)$, is
defined as follows:
$$\reg (M) = \max\{j-i~|~\beta_{i,j}(M)\neq 0 \}.$$

The {\bf projective dimension} of an $R$-module $M$,
denoted by $\pd (M)$, is the length of the minimal free resolution
of $M$, that is,
$$\pd (M) = \max\{i~|~\beta_{i,j}(M)\neq 0~\text{for some}~j\}.$$

Katzman provided the following result on the regularity of
$R/I(G)$.

\begin{lem} \label{lem:Katz}(\cite[Lemma 2.2]{K}) For any graph $G$, $\reg (R/I(G))\geq
a(G)$.
\end{lem}

A simplicial complex $\Delta$ is called \textit{shellable} if the facets
can give a linear order $F_1 ,  . . .  , F_s$ such that for all
$1\leq i<j\leq s$, there exists some $v\in F_j\setminus F_i$ and
some $l\in \{1 ,  . . .  , j-1\}$ with $F_j\setminus F_l=\{v\}$,
\cite{BW}. A \textit{graph} $G$ is called \textit{shellable}, if the simplicial
complex $\Delta_G$ is a shellable simplicial complex.

Let $F\in\Delta$ be a face of $\Delta$. The link of $F$ is the
simplicial complex
$$\textnormal{lk}_\Delta(F)=\{F'\in\Delta |F'\cap F=\emptyset~,~F'\cup F\in \Delta\}$$
and the deletion of $F$ is the simplicial complex
$$\textnormal{del}_\Delta(F)=\{F'\in\Delta~|~F'\cap F=\emptyset\}.$$

If $\Delta$ is a simplicial complex with facets $F_1, . . . ,F_t$,
we denote $\Delta$ by $\langle F_1, . . . ,F_t \rangle$, and
$\{F_1, . . . ,F_t\}$ is called the facet set of $\Delta$. The
facet ideal of $\Delta$ is the square-free monomial ideal
$$\mathcal{F}(\Delta)=(x_F~|~F~\text{is a facet of}~\Delta).$$

A simplicial complex $\Delta$ on the vertex set $V=\{ x_1 ,  . . .
, x_n\}$ is called \textit{vertex decomposable} if either:
\begin{itemize}
\item[{\bf (i)}] $\Delta=\langle\{ x_1 ,  . . .
, x_n\}\rangle$, or $\Delta=\emptyset.$
\item[{\bf (ii)}] There exists some $x\in V$ such that
$\textnormal{lk}_\Delta(\{x\})$ and $\textnormal{del}_\Delta(\{x\})$ are vertex
decomposable, and every facet of $\textnormal{del}_\Delta(\{x\})$ is a facet
of $\Delta$.
\end{itemize}

A graph $G$ is vertex decomposable if the simplicial complex
$\Delta_G$ is vertex decomposable. It is known that a graph $G$ is vertex decomposable if and
only if its connected components are vertex decomposable.

For $S\subseteq V(G)$ we denote by $G\setminus S$ the subgraph of
$G$ obtained by removing all vertices of $S$ from $G$. Moreover, for
any $x\in V(G)$ we denote by $N_G(x)$ the neighbor set of $x$
in $G$, i.e. $N_G(x)=\{y\in V(G)~|~ xy\in E(G)\}$. The
following lemma will be crucial in the proof of our main results.

\begin{lem} \label{lm:ver} (\cite[Lemma 4.2]{DE}) Let $G$ be a graph and suppose
that $x,y\in V(G)$ are two vertices such that $\{x\}\cup
N_G(x)\subseteq \{y\}\cup N_G(y)$. If $G\setminus\{y\}$ and
$G\setminus (\{y\}\cup N_G(y))$ are both vertex decomposable,
then $G$ is vertex decomposable.
\end{lem}

\begin{rem}\em Let $I$ be a square-free monomial ideal and $\Delta$
be a simplicial complex such that $I=\mathcal{F}(\Delta)$. Then
the Alexander dual of $I$ is the ideal
$$I^\vee=(x_F~|~F~is~a~minimal~vertex~cover~of~\Delta).$$
\end{rem}
Notice that $I(G)^\vee$, the cover ideal of a graph $G$ in
Definition \ref{def:cover}, is the Alexander dual of the edge ideal $I(G)$.\\

We require the following result of Terai \cite {T} about the
Alexander dual of a square-free monomial ideal.

\begin{thm}\label{th:dual} Let $I$ be a square-free monomial ideal.
Then $\reg(R/I)=\pd(I^\vee)$.
\end{thm}

\section{Cohen-Macaulay case}
Throughout this section let $G=(V(G),E(G))$ be a very well-covered graph with $2n$
vertices. Hence we may assume $(\ast)$. In this section we first prove that $G$ is Cohen-Macaulay if and only if
it is vertex decomposable. Moreover we show that if $G$ is
Cohen-Macaulay, then the regularity of $R/I(G)$ is equal to the maximum
number of pairwise 3-disjoint edges of $G$. We will use this result to prove the main theorem of the next section.

\begin{thm} \label{th:equiv} Let $G$ be a very well-covered graph with $2n$ vertices. Then the following conditions are equivalent:
\begin{itemize}
\item[{\bf (1)}] $G$ is Cohen-Macaulay.
\item[{\bf (2)}] $G$ is pure shellable.
\item[{\bf (3)}] $G$ is pure vertex decomposable.
\end{itemize}
\end{thm}

\begin{proof} $(3)\Rightarrow (2)\Rightarrow (1)$ always hold for any graph $G$. So it suffices to prove $(1)\Rightarrow (3)$.

We prove the assertion by induction on $n$. If $n=1$, then $G$ is
just an edge and there is nothing to prove. So suppose $n>1$. By Theorem \ref{th:CM} we
may assume $(*)$,$(**)$, $\textnormal{deg}(y_1)=1$ and $x_1y_1\in E(G)$. \\
We have that $\{y_1\}\cup N_G(y_1)\subseteq \{x_1\}\cup N_G(x_1)$. By Lemma \ref{lm:ver} it is enough
to show that $G\setminus \{x_1\}$ and $G\setminus (\{x_1\}\cup
N_G(x_1))$ are vertex decomposable. It is clear that $G\setminus
\{x_1,y_1\}$ has even number of vertices which are not isolated
with $\heit(I(G\setminus \{x_1,y_1\}))=n-1$. It follows from Theorem
1.2 that $G\setminus \{x_1,y_1\}$ is Cohen-Macaulay. Now induction
hypothesis implies that $G\setminus \{x_1,y_1\}$ is vertex
decomposable. Since $\{y_1\}$ is isolated, $G\setminus \{x_1\}$
is vertex decomposable. \\
Now we show that $G\setminus (\{x_1\}\cup
N_G(x_1))$ is vertex decomposable. We first prove the following
claims.\\
{\bf Claim 1.} If $x_t \in N_G(x_1)$, then $y_t$ is isolated in
$G\setminus (\{x_1\}\cup N_G(x_1))$.\\
{\bf Claim 2.} If $y_t \in N_G(x_1)$, then $x_t$ is isolated
in $G\setminus (\{x_1\}\cup N_G(x_1))$.\\\\
\textit{Proof of Claim 1.} Suppose $x_t \in N_G(x_1)$. If $y_t$ is not isolated in $G\setminus (\{x_1\}\cup N_G(x_1))$, then there exists an integer $k$ such that $x_ky_t\in E(G\setminus (\{x_1\}\cup N_G(x_1)))$. From Theorem \ref{th:CM}, (3),(i), we get that $x_1x_k\in E(G)$ and hence $x_k\in N_G(x_1)$. This implies that $x_k\notin V(G\setminus (\{x_1\}\cup N_G(x_1)))$ but $x_ky_t\in E(G\setminus (\{x_1\}\cup N_G(x_1)))$ which is impossible.\\
\textit{Proof of Claim 2.} Suppose $y_t \in N_G(x_1)$ but $x_t$ is not isolated in
$G\setminus (\{x_1\}\cup N_G(x_1))$. If $x_kx_t\in
E(G\setminus (\{x_1\}\cup N_G(x_1)))$ for some $k$, then we get
$x_1x_k\in E(G)$ and so $x_k \in N_G(x_1)$, a contradiction.\\
If $x_ty_k\in E(G\setminus (\{x_1\}\cup N_G(x_1)))$ for
some $k$, then we must have $x_1y_k\in E(G)$ and hence $y_k
\in N_G(x_1)$. This shows that $y_k\notin V(G\setminus
(\{x_1\}\cup N_G(x_1)))$ but $x_ty_k\in E(G\setminus
(\{x_1\}\cup N_G(x_1)))$ which is impossible.

The above statements show that
$$H=(G\setminus (\{x_1\}\cup N_G(x_1)))\setminus\{\text{isolated vertices of }G\setminus (\{x_1\}\cup N_G(x_1))\}$$
has even number of vertices which are not isolated and its height
is half of the number of vertices. It follows from Remark \ref{rem:equiv}
that $H$ is Cohen-Macaulay and so it is vertex decomposable by
induction. Therefore $G\setminus (\{x_1\}\cup N_G(x_1))$ is vertex
decomposable.
\end{proof}

Now we study the Castelnuovo-Mumford regularity of a Cohen-Macaulay very well-covered graphs with $2n$ vertices. Since this type of graphs contains the set of Cohen-Macaulay bipartite graphs, our result generalizes the
same well-known result on Cohen-Macaulay bipartite graphs.

\begin{thm} \label{thm:reg} Let $G$ be a very well-covered graph with $2n$ vertices. If $G$ is Cohen-Macaulay, then $\reg(R/I(G))=a(G)$.
\end{thm}

\begin{proof} By Lemma \ref{lem:Katz}  and Theorem \ref{th:dual} it is enough to show that $\pd(I(G)^\vee)\leq a(G)$.
We proceed by induction on $n$. If $n=1$, then $G$ is single edge
$x_1y_1$ and $(I(G)^\vee)=(x_1,y_1)$. Therefore
$\pd(I(G)^\vee)=1=a(G)$. Now suppose $n>1$. By Theorem 1.2 we may
assume $\textnormal{deg}(y_1)=1$, $N_G(y_1)=\{x_1\}$, and
$N_G(x_1)=\{x_{i_1}, . . . ,x_{i_k},y_1,y_{j_1}, . . . ,y_{j_s}\}$
with $\{i_1, . . . ,i_k\}\cap \{1,j_1, . . . ,j_s\}=\emptyset$. Note
that there is no minimal vertex cover of $G$ containing  both
$x_1$ and $y_1$ and that any minimal vertex cover of $G$ not
containing $x_1$ must contain $N_G(x_1)$. Set $G'=G\setminus
(\{x_1\}\cup N_G(x_1))$ and $G^{''}=G\setminus (\{y_1\}\cup
N_G(y_1))$. Let $(I(G')^\vee)$ and
$(I(G^{''})^\vee)$ ideals of
$R=K[x_1, . . . ,x_n,y_1, . . . ,y_n]$, then using the same arguments as in \cite[Theorem 3.3]{VT}, we have
\begin{itemize}
\item[{\bf (1)}] $I(G)^\vee=x_1I(G^{''})^\vee+x_{i_1}\cdots x_{i_k}y_1y_{j_1}\cdots y_{j_s}I(G')^\vee.$
\item[{\bf (2)}] $x_1I(G^{''})^\vee\bigcap x_{i_1}\cdots x_{i_k}y_1y_{j_1}\cdots y_{j_s}I(G')^\vee=x_1x_{i_1}\cdots x_{i_k}y_1y_{j_1}\cdots y_{j_s}I(G')^\vee.$
\end{itemize}
The above statements imply that there is an exact sequence
\begin{equation*}\begin{split}
0\longrightarrow x_1x_{i_1}\cdots x_{i_k}&y_1y_{j_1}\cdots y_{j_s}I(G')^\vee \longrightarrow\\
&x_1I(G^{''})^\vee \oplus ~x_{i_1}\cdots x_{i_k}y_1y_{j_1}\cdots
y_{j_s}I(G')^\vee \longrightarrow I(G)^\vee \longrightarrow 0.
\end{split}\end{equation*}
The above exact sequence yields
\begin{equation*}\begin{split}
\pd(I(G)^\vee)\leq\max\{&\pd(x_1x_{i_1}\cdots x_{i_k}y_1y_{j_1}\cdots y_{j_s}I(G')^\vee)+1,\\
&\pd(x_1I(G^{''})^\vee), \pd(x_{i_1}\cdots x_{i_k}y_1y_{j_1}\cdots
y_{j_s}I(G')^\vee)\},
\end{split}\end{equation*}
Note that for any monomial ideal $I$ and monomial $f$ with
property that $\textnormal{supp}(f)\cap \textnormal{supp}(g)=\emptyset$, for all $g\in
\mathcal{G}(I)$ (minimal generating set of $I$), we have
$\pd(fI)=\pd(I)$. Therefore
$$\pd(I(G)^\vee)\leq\max \{\pd(I(G')^\vee)+1,\pd(I(G^{''})^\vee)\}.$$
As explained in the proof of Theorem 3.1, $G'\setminus
\{\text{isolated vertices of }G'\}$ and $G^{''}$ have an even number of
vertices which are not isolated and their heights are half of
the number of vertices. Since isolated vertices do not affect on
$\reg(R/I(G))$ and $a(G)$ for any graph $G$, our induction implies
that $\pd(I(G')^\vee)+1\leq a(G')+1$ and $\pd(I(G^{''})^\vee)\leq
a(G^{''}).$ One can see that $a(G^{''})\leq a(G)$ and
$a(G')+1\leq a(G)$ (adding the edge $x_1y_1$ to any pairwise
3-disjoint set of edges in $G'$ is a set of pairwise 3-disjoint
edges in $G$ ). Therefore $\pd(I(G)^\vee)\leq a(G).$
\end{proof}

\section{Regularity in unmixed case}

Let $\mathfrak{d}$ be a semidirected graph (it has both directed and undirected edges) on [n]. We will write $j\succ i$ if there is a directed path from i to j in $\mathfrak{d}$. By $j\succcurlyeq i$ (and, equivalently, $i\preccurlyeq j$) we mean that $j\succ i$ or $j=i$.  For $A\subseteq [n]$, we say that $j\succcurlyeq A$ if there exists $i\in A$ such that $j\succcurlyeq i$. We say that a set $A\subseteq [n]$ is an \textit{antichain} if for all $i,j\in A$, there is no directed path from $i$ to $j$ in $\mathfrak{d}$, and, by $\mathcal{A}_{\mathfrak{d}}$, denote the set of antichains in $\mathfrak{d}$. We consider $\emptyset$ as an antichain. We say that $\mathfrak{d}$ is \textit{acyclic} if there are no directed cycles in $\mathfrak{d}$, and \textit{transitively closed} if, for all distinct $i,j,k\in [n]$, whenever $ij$ (from $i$ to $j$) and $jk$ (from $j$ to $k$) are directed edges in $\mathfrak{d}$, $ik$ (from $i$ to $k$) is again a directed edge in $\mathfrak{d}$, and, whenever $ij$ is an undirected edge and $kj$ (from $k$ to $j$) is a directed edge, $ik$ is an undirected edge in $\mathfrak{d}$. \\
Let $G$ be a graph with $2n$ vertices which are not isolated with $\heit(I(G))=n$ and suppose $G$ satisfies $(\ast)$. We associate to $G$ a semidirected graph $\mathfrak{d}_G$ on [n] defined as follows: for $i\neq j\in [n]$, $ij$ is a directed edge of $\mathfrak{d}_G$ from $i$ to $j$ if
and only if $x_iy_j$ is an edge of $G$, and, $ij$ is an undirected edge of $\mathfrak{d}_G$ if and only if  $x_ix_j$ is an edge of $G$. Notice that if $G$ is unmixed (in particular, $G$ is very well-covered), $\mathfrak{d}_G$ is simple, i.e., without loops and multiple edges.

In the next lemma, $\textnormal{Unm}(R/I)$ denotes the set of associated prime ideals $\mathfrak{p}$ of $I$ such that $\heit(\mathfrak{p})=\heit(I)$.

\begin{lem} \label{lem:prime} Let $G$ be a graph with $2n$ vertices which are not isolated, $\heit(I(G))=n$, and suppose $G$ satisfies $(\ast)$.  For all $\mathfrak{p}\in \textnormal{Unm}(R/I)$, if $y_i\in \mathfrak{p}$ and $j\succcurlyeq i$, then $y_j\in \mathfrak{p}$, $i, j \in [n]$.
\end{lem}
\begin{proof} We proceed by induction on the length of the directed path from $i$ to $j$. Without loss of generality, we may assume $ij$ is a directed edge in $\mathfrak{d}_G$. Now let $k\in[n]$. Since $x_ky_k\in I\subseteq \mathfrak{p}$, $x_k\in \mathfrak{p}$ or $y_k\in \mathfrak{p}$, But since $\heit(\mathfrak{p})=n$ we have that $x_k\in \mathfrak{p}$ if and only if $y_k\notin \mathfrak{p}$. Now $y_i\in \mathfrak{p}$ implies that $x_i\notin \mathfrak{p}$ which together with $x_iy_j\in\mathfrak{p}$ shows that $y_j\in\mathfrak{p}$.
\end{proof}

Let $\mathfrak{d}$ be a semidirected graph. We say that a pair $i,j$ of vertices $\mathfrak{d}$ are
\textit{strongly connected} if there are directed paths from $i$ to $j$ and from $j$ to $i$. A \textit{strong component} of $\mathfrak{d}$ is an induced subgraph maximal under the property that every pair of vertices in it is strongly connected. Strong components of $\mathfrak{d}$ form a partition of its vertex set. \\

We state the following definition.
\begin{defn}\label {def:new1}
Let $G$ be a graph with $2n$ vertices which are not isolated with $\heit(I(G))=n$ and suppose $G$ satisfies $(\ast)$. Let $\mathcal{Z}_1, \ldots,\mathcal{Z}_t$ be the vertex sets of the strong components of
$\mathfrak{d}_G$. We define a semidirected graph $\mathfrak{\widehat{d}}_G$ on $[t]$ by setting, for $a\neq b\in [t]$,

$\bullet$ $ab$ to be a directed edge (from $a$ to $b$) if there exists a directed path in $\mathfrak{d}_G$ from any  of the vertices in $\mathcal{Z}_a$ to any of the vertices in $\mathcal{Z}_b$;

$\bullet$ $ab$ to be an undirected edge if there exists an undirected edge in $\mathfrak{d}_G$ from any  of the vertices in $\mathcal{Z}_a$ to any of the vertices in $\mathcal{Z}_b$.
\end{defn}

It is known that $\mathfrak{\widehat{d}}_G$ has no directed cycles. Moreover, if $G$ is unmixed, then $\mathfrak{d}_G$ and therefore $\mathfrak{\widehat{d}}_G$ are transitively closed.

We will use the same notation as in Definition \ref{def:new1} for the induced order, i.e., say that $b\succ a$ if there is a directed edge from $a$ to $b$.

Hence we give the following.
\begin{defn}\label{def:acyc}
Let $G$ be a graph with $2n$ vertices which are not isolated with $\heit(I(G))=n$ and suppose $G$ satisfies $(\ast)$. Let $\mathcal{Z}_1, \ldots,\mathcal{Z}_t$ be the vertex sets of the strong components of
$\mathfrak{d}_G$. We define acyclic reduction of $G$ the graph $\widehat{G}$ on new vertices $\{u_1, \ldots, u_t\}\cup\{v_1,\ldots , v_t\}$, with edges

$\bullet$ $u_av_a$, for all $1\leq a\leq t$;

$\bullet$ $u_av_b$, for all directed edges $ab\in \mathfrak{\widehat{d}}_G$;

$\bullet$ $u_au_b$, for all undirected edges $ab\in \mathfrak{\widehat{d}}_G$.

\end{defn}

Let $G$ be a graph with $2n$ vertices which are not isolated with $\heit(I(G))=n$ and suppose $G$ satisfies $(\ast)$. For any antichain $A$ of $\mathfrak{\widehat{d}}_G$ we define
$$\Omega_A=\{j\in \mathcal{Z}_b ~|~ b\succcurlyeq A\}.$$
Since, for any antichain $A=\{i_1, \ldots,i_r\}$ of $\mathfrak{d}_G$, there exists a unique antichain $A'=\{a_1, \ldots ,a_r\}$ in $\mathfrak{\widehat{d}}_G$ such that $i_j\in \mathcal{Z}_{a_j}$ for all $1\leq j \leq r$, we set
$$\Omega_A=\Omega_{A'}.$$

\begin{lem} \label{lem:un} Let $G$ be a graph with $2n$ vertices which are not isolated, $\heit(I(G))=n$, and suppose $G$ satisfies $(\ast)$. Then
$$a(G)\geq \max\{|A|~|~ A\in\mathcal{A}_{\mathfrak{d}_G}~,~ \Omega_A\nsupseteq e ~for ~all~undirected~ edges ~e~ in ~ \mathfrak{d}_G\}.$$
\end{lem}
\begin{proof} Let $A=\{i_1, \ldots ,i_r\}\in \mathcal{A}_{\mathfrak{d}_G}$ be such that $\Omega_A\nsupseteq e$ for all undirected edges $e$ in $\mathfrak{d}_G$.
Since  $A$ is an antichain, hence $\mathfrak{d}_G$ has no directed edges $i_li_s$ and $i_si_l$, and since $\Omega_A\nsupseteq e$ for all undirected edges $e$ in $\mathfrak{d}_G$, therefore $\mathfrak{d}_G$ does not contain the undirected edge $i_li_s$ for all $i_l\neq i_s\in A$. Finally, since $\Omega_A\supseteq A$, then $A\nsupseteq e$ and therefore $\{x_iy_i~|~ i\in A\}$ is a set of pairwise 3-disjoint edges in $G$.
\end{proof}
\begin{lem} \label{lem:well} Let $G$ be a very well-covered graph. Then $\widehat{G}$ is a Cohen-Macaulay very well-covered graph.
\end{lem}

\begin{proof} By definition $\widehat{G}$ has an even number of vertices which are not isolated and with $2\heit(I(\widehat{G}))=|V(\widehat{G})|$. By Remark \ref{rem:equiv}, it is enough to show that $\widehat{G}$ satisfies $(\ast), (\ast \ast)$, (i), and (ii). Clearly $\widehat{G}$ satisfies $(\ast)$. Since $\mathfrak{\widehat{d}}_G$ is acyclic, its vertex set can be relabeled such that every directed edge of $\mathfrak{\widehat{d}}_G$ is of the form $ij$ with $i<j$. This shows that $\widehat{G}$ satisfies the condition $(\ast\ast)$. Since $G$ is unmixed, $\mathfrak{d}_G$ has no multiple edge and is transitively closed. Therefore $\mathfrak{\widehat{d}}_G$ has no multiple edge. In fact, if there are both undirected and directed edges from $a$ to $b$ in $\mathfrak{\widehat{d}}_G$, then there exist $i_1,i_2\in\mathcal{Z}_a$ and  $j_1,j_2\in\mathcal{Z}_b$ such that $i_1j_1$ is an undirected and $i_2j_2$ is a directed edge in $\mathfrak{d}_G$. Since $\mathfrak{d}_G$ is transitively closed, the directed edge $i_1j_1$ must belong to $\mathfrak{d}_G$, contradicts the fact that $\mathfrak{d}_G$ has no multiple edge. Therefore $\widehat{G}$ satsfies the condition (ii). Finally $\widehat{G}$ satisfies (i) since $\mathfrak{\widehat{d}}_G$ is transitively closed and $G$ satisfies (ii).
\end{proof}

\begin{rem} \em Let $G$ be a very well-covered graph. It is easy  to see that $\mathfrak{\widehat{d}}_G=\mathfrak{d}_{\widehat{G}}$. Moreover, if $G$ is itself Cohen-Macaulay, then $G=\widehat{G}$.
\end{rem}

\begin{lem} \label{lem:ass} Let $G$ be a very well-covered graph with $2n$ vertices. Then
\begin{equation*}\begin{split}
\textnormal{Ass}(R/I)=\{(x_i~|~i\notin\Omega_A)+(y_i~|~i\in\Omega_A)~ ~|~ ~ &A\in\mathcal{A}_{\mathfrak{\widehat{d}}_G}~,~ \Omega_A\nsupseteq e \\ &for ~all~undirected~ edges ~e~ in ~ \mathfrak{d}_G \}.
\end{split}\end{equation*}
\end{lem}
\begin{proof} Let $\mathfrak{p}\in \textnormal{Ass}(R/I)$. Since $I$ is unmixed, just one of $x_i$ and $y_i$ belongs to $\mathfrak{p}$ for all $i=1, \ldots, \heit(I)$. Let $U=\{b~|~y_j\in\mathfrak{p}~\text{for some } j\in\mathcal{Z}_b\}$. Since $\mathcal{Z}_1, \ldots,\mathcal{Z}_t$ are the vertex sets of the strong components of $\mathfrak{d}_G$, from Lemma \ref{lem:prime} it follows that $y_j\in\mathfrak{p}$ for all $j\in\cup_{b\in U}\mathcal{Z}_b$, and that if $b'\succ b$ and $b\in U$, then $b'\in U$. Suppose $A$ is the set of minimal elements of $U$ under $\succ$. One can see that $A$ is an antichain in $\mathfrak{\widehat{d}}_G$, $U=\{b~|~b\succcurlyeq A\}$, and $\Omega_A=\cup_{b\in U}\mathcal{Z}_b=\{j~|~y_j\in\mathfrak{p}\}$. Now we show that $\Omega_A$ does not contain any undirected edge of $\mathfrak{d}_G$. Suppose the contrary that $e=\{i,j\}\subseteq\Omega_A$ is an undirected edge in $\mathfrak{d}_G$. So that $x_ix_j\in I\subseteq\mathfrak{p}$ and hence we may assume $x_i\in \mathfrak{p}$. Therefore $y_i\notin \mathfrak{p}$. On the other hand, since $i\in\Omega_A$, we get that $y_i\in\mathfrak{p}$ a contradiction. Hence $\textnormal{Ass}(R/I)\subseteq\{(x_i~|~i\notin\Omega_A)+(y_i~|~i\in\Omega_A)~ ~|~ ~ A\in\mathcal{A}_{\mathfrak{\widehat{d}}_G}~,~ \Omega_A\nsupseteq e ~\text{for all undirected edges }e~ \text{of} ~ \mathfrak{d}_G \}$.

Conversely, let $A\in\mathcal{A}_{\mathfrak{\widehat{d}}_G}$ be such that $\Omega_A$ does not contain any undirected edge of $\mathfrak{d}_G$ and let $\mathfrak{p}=(x_i~|~i\notin\Omega_A)+(y_i~|~i\in\Omega_A)$. Therefore $\heit(\mathfrak{p})=\heit(I)$. Since $I$ is unmixed, it suffices to prove that $I\subseteq \mathfrak{p}$. $I$ is generated by monomials of the forms $x_iy_i$ ($i=1, \ldots ,n$), $x_iy_j$, and $x_ix_j$ for some $1\leq i\neq j\leq n$. It is clear that $x_iy_i\in\mathfrak{p}$ for all $i=1, \ldots ,n$. So assume $i\neq j$. First let $x_iy_j\in I$. If $i\notin \Omega_A$, there is nothing to prove. If $i\in\Omega_A$, then there exists $a,b,b'$ such that $a\in A$, $b\succcurlyeq a$, $i\in \mathcal{Z}_b$, and $j\in\mathcal{Z}_{b'}$. Since $ij$ is an directed edge in $\mathfrak{d}_G$, we get that $b'\succ b$ in $\mathfrak{\widehat{d}}_G$. Hence $b'\succ a$, and $j\in\Omega_A$ which shows that $y_j\in\mathfrak{p}$ and so $x_iy_j\in\mathfrak{p}$. Now let $x_ix_j\in I$. Since $\Omega_A$ does not contain any undirected edge of $\mathfrak{d}_G$, we have $\{i,j\}\nsubseteq \Omega_A$. Therefore $i\notin\Omega_A$ or $j\notin\Omega_A$ which shows that $x_i\in \mathfrak{p}$ or $x_j\in \mathfrak{p}$. Hence $x_ix_j\in\mathfrak{p}$.
\end{proof}

With the same notations used in Definion \ref{def:acyc}, let $\widehat{G}$ be the acyclic reduction of $G$ with edge ideal $\widehat{I}=I(\widehat{G})$ as an ideal of $S=k[u_1,\ldots , u_t,v_1, \ldots , v_t]$.
Thanks to Lemma \ref{lem:ass}, the next result is a generalization of Remark 3.3 in \cite{Ku}. Since the proof follows by similar arguments as in \cite[Remark 3.3]{Ku}, we omit it.

\begin{prop}\label{prop:red}Let $G$ be a very well-covered graph. Let $\widehat{G}$ be the acyclic reduction of $G$ with edge ideal $\widehat{I}\subseteq S$. Then $\reg(R/I(G))=\pd((\widehat{I})^\vee)=\reg(S/\widehat{I})$.
\end{prop}

\begin{lem} \label{lem:acy}Let $G$ be a very well-covered graph. Also assume $\widehat{G}$ be the acyclic reduction of $G$. Then
\begin{equation*}\begin{split}
&\max\{|A|~|~ A\in\mathcal{A}_{\mathfrak{d}_G}~,~ \Omega_A\nsupseteq e ~for ~all~undirected~ edges ~e~ in ~ \mathfrak{d}_G\}=\\
&=\max\{|A|~|~ A\in\mathcal{A}_{\mathfrak{d}_{\widehat{G}}}~,~ \Omega_A\nsupseteq e ~for ~all~undirected~ edges ~e~ in ~ \mathfrak{d}_G\}.
\end{split}\end{equation*}
\end{lem}
\begin{proof} Let $A\in\mathcal{A}_{\mathfrak{d}_G}$ be such that $\Omega_A$ does not contain any undirected edge of $\mathfrak{d}_G$ and suppose $A=\{i_1, \ldots ,i_r\}$ with $i_j\in \mathcal{Z}_{a_j}$ for all $1\leq j \leq r$. Let $A'=\{a_1, \ldots ,a_r\}$. It is easy to see that $A'\in\mathcal{A}_{\mathfrak{d}_{\widehat{G}}}$. Hence it suffices to show that $\Omega_{A'}$ does not contain any undirected edge of $\mathfrak{d}_G$. Suppose the contrary that $\alpha\beta$ be an undirected edge of $\mathfrak{d}_G$ which is also in $\Omega_{A'}$. Therefore there exist $b,b'\succcurlyeq A'$ such that $\alpha\in\mathcal{Z}_b$ and $\beta\in\mathcal{Z}_{b'}$. Hence, there exist $a_p,a_q\in A'$ such that $b\succcurlyeq a_p$ and $b'\succcurlyeq a_q$. Now $\alpha\in\mathcal{Z}_b$, $\beta\in\mathcal{Z}_{b'}$, $i_p\in\mathcal{Z}_{a_p}$ and $i_q\in\mathcal{Z}_{a_q}$ imply that there are the directed edges $i_p\alpha$ and $i_q\beta$ in $\mathfrak{d}_G$. Since $G$ is unmixed, we get that
$$x_{i_q}y_\beta~,~x_\beta x_\alpha\in E(G)~\Longrightarrow ~x_{i_q}x_\alpha\in E(G)$$
$$x_{i_p}y_\alpha~,~x_\alpha x_{i_q}\in E(G)~\Longrightarrow ~x_{i_p}x_{i_q}\in E(G)$$
which contradicts the definition of $A$.

Conversely, let $A'\in\mathcal{A}_{\mathfrak{d}_{\widehat{G}}}$ be such that $\Omega_{A'}$ does not contain any undirected edge of $\mathfrak{d}_G$ and let $A'=\{a_1, \ldots ,a_r\}$. Also suppose $A=\{i_1, \ldots ,i_r\}$ with $i_j\in \mathcal{Z}_{a_j}$ for all $1\leq j \leq r$. Clearly $A$ is an antichain in $\mathfrak{d}_G$. It follows that $\Omega_A=\Omega_{A'}$ and hence $\Omega_A$ does not contain any undirected edge of $\mathfrak{d}_G$. This completes the proof.
\end{proof}

\begin{thm} Let $G$ be a very well-covered graph with $2n$ vertices. Then
$$\reg(R/I(G))=\max\{|A|~|~ A\in\mathcal{A}_{\mathfrak{d}_G}~,~ \Omega_A\nsupseteq e ~for ~all~undirected~ edges ~e~ in ~ \mathfrak{d}_G\} = a(G).$$
\end{thm}
\begin{proof} Let $\widehat{G}$ be the acyclic reduction of $G$ on the vertex set $\{u_1, \ldots , u_t\}\cup\{v_1, \ldots , v_t\}$ with edge ideal $\widehat{I}\subseteq S$. Since $\widehat{G}$ is Cohen-Macaulay, from Proposition \ref{prop:red} and Lemma \ref{lem:acy}, it is enough to prove the assertion in the case when $G$ is Cohen-Macaulay. So suppose $G$ is a Cohen-Macaulay very well-covered graph with $2n$ vertices. We may assume $(\ast)$.  First of all observe that under our assumption, from Theorem \ref{thm:reg}, $\reg(R/I(G)) =a(G)$. \\
Now let $B$ be a set of pairwise 3-disjoint edges in $G$. Set
$$A=\{x_i~|~x_iy_j\in B~ \text{for some} ~j\}\cup\{x_i~|~x_ix_k\in B~,~i<k\}.$$
One can see that $A$ is an antichain in $\mathfrak{d}_G$ and that $\Omega_A$ does not contain any undirected edge of $\mathfrak{d}_G$. This implies that
$$a(G)\leq \max\{|A|~|~ A\in\mathcal{A}_{\mathfrak{d}_G}~,~ \Omega_A\nsupseteq e ~\text{for all undirected edges} ~e~ \text{of} ~ \mathfrak{d}_G\}$$
which together with Theorem \ref{thm:reg} and Lemma \ref{lem:un} completes the proof.

\end{proof}
It was suggested by Villarreal that if $G$ is a Cohen-Macaulay graph, then $G\setminus\{v\}$ is Cohen-Macaulay
for some vertex $v$ in G, see \cite{V1}. Estrada and Villarreal proved this for those graphs that are Cohen-Macaulay and bipartite  by showing the fact that there is a vertex $v\in V(G)$ such that $\textnormal{deg}(v)=1$ (\cite[Theorem 2.4]{EV}). Van Tuyl and Villarreal proved the same result for sequentially Cohen-Macaulay bipartite graphs in \cite [Lemma 3.9]{VTV}). Using the above fact, Van Tuyl in \cite{VT} showed that if $G$ is bipartite, then

$\bullet$ $G$ is sequentially Cohen-Macaulay if and only if it is vertex decomposable.

$\bullet$ If $G$ is sequentially Cohen-Macaulay, then $\reg(R/I(G))=a(G)$.\\
So it is natural to ask the following question:
\begin{ques} Let $G$ be a sequentially Cohen-Macaulay graph with $2n$ vertices which are not isolated and with $\heit(I(G))=n$. Does $G$ have a vertex $v$ such that $\textnormal{deg}(v)=1?$
\end {ques}
If the answer is positive, one can easily generalize main results of \cite{VT} to the class of graphs in Question 4.11.

\par \vspace{2mm}

{\bf Acknowledgements.}
We are grategul for helpful comments of Professor A.~Van Tuyl.

\end{document}